\documentclass[11pt]{article}
\usepackage[final]{epsfig}
\usepackage[left=3cm,right=3cm, top=2.5cm,bottom=2.5cm,bindingoffset=0cm]{geometry}
\usepackage{graphics}
\usepackage{amsmath}
\usepackage{amsfonts}
\usepackage{latexsym}
\usepackage{amssymb}
\usepackage{amsthm}
\usepackage{graphicx}
\usepackage{epstopdf}
\usepackage{multicol,multirow}
\usepackage{hyperref}
\usepackage{mathtools}

\usepackage{mathrsfs}

\usepackage{relsize}

\usepackage{amsthm}

\usepackage[latin1]{inputenc}
\usepackage{tikz}
\usetikzlibrary{shapes,arrows,decorations.pathmorphing}
\usepackage{float}
\makeatletter
\def\thmhead@plain#1#2#3{%
  \thmname{#1}\thmnumber{\@ifnotempty{#1}{ }\@upn{#2}}%
  \thmnote{ {\the\thm@notefont#3}}}
\let\thmhead\thmhead@plain

\makeatother

\newcounter{AppCounter}

\def\restrict#1{\raise-.5ex\hbox{\ensuremath|}_{#1}}

\newtheorem{lemma}{Lemma}[section]
\newtheorem{proposition}[lemma]{Proposition}

\newtheorem{remark-definition}[lemma]{Remark-Definition}
\newtheorem{theorem}[lemma]{Theorem}
\newtheorem{corollary}[lemma]{Corollary}

\newtheorem{proposition-conjecture}[lemma]{Proposition-conjecture}

\theoremstyle{definition}

\newtheorem{definition}[lemma]{Definition}
\newtheorem{remark}[lemma]{Remark}

\newcommand{\marginnote}[1]
{%\mbox{}\marginpar{\center{\hspace{0pt}\tiny{\bf#1}}}
}

\newcounter{cy}

\newcounter{bk}

\newcounter{dps}

\title{The helicity uniqueness conjecture in 3D hydrodynamics}

\author{Boris Khesin\thanks{Department of Mathematics, University of Toronto, Toronto, ON M5S 2E4, Canada; e-mail:  \tt{khesin@math.toronto.edu}},
Daniel Peralta-Salas\thanks{Instituto de Ciencias Matem\'{a}ticas, Consejo Superior de Investigaciones Cient\'{i}ficas,
28049 Madrid, Spain; e-mail:  \tt{dperalta@icmat.es}},
  and Cheng Yang\thanks{Department of Mathematics and Statistics, McMaster University, Hamilton, ON L8S 4K1, Canada, and the Fields Institute for Research in Mathematical Sciences, Toronto, ON M5T 3J1, Canada;
 e-mail: \tt{yangc74@math.mcmaster.ca}
  }}

\date{}

\begin{document}

\maketitle
\begin{abstract}
We prove that the helicity is the only regular Casimir function for
the coadjoint action of the volume-preserving diffeomorphism group $\text{SDiff}(M)$ on smooth exact divergence-free
vector fields on a closed three-dimensional manifold  $M$. More precisely, any regular $C^1$ functional defined on the space of $C^\infty$
(more generally, $C^k$, $k\ge 4$) exact divergence-free vector fields and invariant under arbitrary
volume-preserving diffeomorphisms can be expressed as a $C^1$ function of the helicity.
This gives a complete description of Casimirs for adjoint and coadjoint actions of $\text{SDiff}(M)$ in 3D
and completes the proof of Arnold-Khesin's 1998 conjecture for a manifold $M$ with trivial first homology group. Our proofs make use of different tools from the theory of dynamical systems, including normal forms for divergence-free vector fields, the Poincar\'e-Birkhoff theorem, and a division lemma for vector fields with hyperbolic zeros.
\end{abstract}

\section{Introduction} \label{intro}

Let $(M,g)$ be a closed (i.e., compact and without boundary) three-dimensional Riemannian manifold with
volume form $d\mu$. The motion of an inviscid and incompressible fluid filling
$M$ is governed by the  {\it Euler equations}:
\begin{equation}\label{idealEuler}
\left\{\begin{array}{l}
\partial_t v+\nabla_v v=-\nabla p\,,\\\
{\rm div}\, v=0\,,
\end{array}\right.
\end{equation}
where $\nabla_v v$ is the Riemannian covariant derivative of the field $v$ along itself, ${\rm div}$ is the divergence operator, and the pressure function $p$ is uniquely determined up to an additive constant.

In the 1960's,  Moreau~\cite{Mo61} and Moffatt~\cite{Mo69} discovered a  conserved quantity for the Euler equations, the {\it helicity}, which is a functional defined for the vorticity vector field $\omega:=\text{curl}\, v$  as follows:
$$
\mathcal H(\omega):=\int_M \omega\cdot \text{curl}^{-1} \omega\,d\mu=\int_M \omega\cdot v\,d\mu\,.
$$
Here the dot $\cdot$ denotes the (pointwise) Riemannian inner product of two fields. The operator $\text{curl}$ is defined via differential forms as $i_{\text{curl}\, v}d\mu=dv^\flat$, where $v^\flat$ is the metric-dual $1$-form of $v$.
The origin of the helicity conservation is Kelvin's law of  the vorticity transport by the flow.
Actually, it is easy to check that $\mathcal H(\Phi_*\omega)=\mathcal H(\omega)$ for any (orientation-preserving) volume-preserving diffeomorphism $\Phi:M\to M$, i.e., the helicity is invariant under arbitrary volume-preserving transformations, rather
than under the specific family of diffeomorphisms defined by a fluid flow. In geometric terms, this makes the helicity a Casimir functional for the coadjoint action of the group of $C^\infty$ volume-preserving diffeomorphisms on exact divergence-free vector fields. Note that the helicity $\mathcal H$ is defined only on exact divergence-free vector fields, since it requires  finding the field-potential, ${\rm curl}^{-1}$.

In the 1998 monograph~\cite[Section I.9]{arkh} it was  conjectured that the helicity is the only Casimir function
for the group $\text{SDiff}(M)$ of $C^\infty$ volume-preserving diffeomorphisms of a closed 3D manifold $M$.
In this paper we prove this conjecture for manifolds $M$ with the trivial first homology group (where all divergence-free vector fields are exact), and under some natural regularity assumptions on the invariants.
Furthermore, we establish similar results for the adjoint action of the group $\text{SDiff}(M)$ and for
$C^k$-fields with $k\geq 4$.
The case of $M$ with $H_1(M, \mathbb R)\not=0$ involves such an invariant of the adjoint action as the rotation number of a vector field, see~\cite{Arn69}, and we shall describe its interrelation with helicity.

To formulate the main result one needs the notion of regular integral invariants (see Definition~\ref{def:regInt1}), which, roughly speaking, means  $C^1$ functionals on the space $\mathfrak X^k_{ex}$ of exact $C^k$ vector fields that are invariant under volume-preserving diffeomorphisms and whose (Fr\'{e}chet) derivative is an integral operator with continuous kernel, see~\cite{EnPeTo}.
% (with respect to the Whitney $C^k$ topology on $\mathfrak X^k_{ex}$).

\begin{theorem}\label{T:main}
Let $\mathcal F $ be a regular integral invariant on the space $\mathfrak X^k_{ex}$  (endowed with the Whitney $C^k$ topology).
Then $\mathcal F $ is a function of the helicity provided that $k\geq 4$ (including the case $k=\infty$), i.e. there exists a $C^1$ function $f: \mathbb R \rightarrow \mathbb R$  such that $\mathcal F (w) =f(\mathcal H(w))$ for any $w\in\mathfrak X_{ex}^k$.
\end{theorem}

\begin{corollary}\label{C:main}
The helicity $\mathcal H$ is the only regular Casimir (i.e. coadjoint invariant) for the group $\text{SDiff}(M)$ of $C^\infty$
volume-preserving diffeomorphisms of a closed 3D manifold $M$ (i.e. the conjecture of {\rm \cite{arkh}} holds) provided that $H_1(M, \mathbb R)=0$.
\end{corollary}

Concerning Corollary~\ref{C:main}, we recall that when $H_1(M)=0$, the dual space to the Lie algebra of the group $\text{SDiff}(M)$ can be identified with the space of
$C^\infty$ exact divergence-free vector fields endowed with the Whitney $C^\infty$ topology (see Section~\ref{sec:pfCor} for details).

An analogous result for $k=1$, i.e., for functionals acting on $\mathfrak X^1_{ex}$, was proved in~\cite{EnPeTo} (see also~\cite{Ku14,Ku16} for the case of manifolds with boundary and divergence-free vector fields admitting a global cross section).
However, none of the results for $C^k$ implies the results for other $k$.
The use of the space $\mathfrak X^1_{ex}$ (endowed with the Whitney $C^1$ topology) is key in the proof of~\cite{EnPeTo}, which is based on the existence of a residual subset of vector fields with special dynamical properties. Indeed, the strategy
for $C^1$-fields makes use of a theorem by M. Bessa~\cite{Besa} showing that there exists a dense set of vector fields
in $\mathfrak X^1_{ex}$ that are topologically transitive; such a dense set of fields with a dense orbit cannot exist when one considers exact divergence-free vector fields of class $C^4$ or higher, as a consequence of the KAM theorem.

For the proof of Theorem~\ref{T:main} for $C^k$ with $k\geq 4$ we use different (and somewhat more elementary) tools from the theory of dynamical systems, including normal forms for divergence-free vector fields, the Poincar\'e-Birkhoff theorem, and a division lemma for vector fields with hyperbolic zeros. The assumption $k\geq 4$ is due to the use of Sard's theorem, which requires sufficiently high regularity of the functions under consideration. In particular, it cannot be applied in the $C^1$ setting, where the corresponding functions are only $C^0$. (Also note that although any $C^1$-functional acting on $\mathfrak X^1_{ex}$ can also be evaluated at elements of $\mathfrak X^k_{ex}$, $k\geq 2$, the kernel of the derivative of such a functional on $\mathfrak X^k_{ex}$ maps $C^k$ vector fields into $\mathfrak X^1_{ex}$ only, and not into $\mathfrak X^k_{ex}$, so the second assumption in Definition~\ref{def:regInt1} below  would not be fulfilled.) The cases with $k=2$ or $3$ remain open.

Given a divergence-free vector field $v$ on $M$ one can define its rotation class $\lambda(v)\in H_1(M, \mathbb R)$, see \cite{Arn69} and Section \ref{sec:pfCor}. Namely, for a three-dimensional manifold one can consider the cohomology class of the closed 2-form $i_v d\mu$ in $H^2(M,\mathbb R)$, which is Poincar\'e isomorphic to $H_1(M, \mathbb R)$.

\begin{corollary}\label{cor:adjoint}
The helicity $\mathcal H$ is the only regular invariant of the adjoint action for the group $\text{SDiff}(M)$
on the space $\mathfrak X^k_{ex}$ ($k\geq 4$) of exact divergence-free vector fields. For
a manifold with nontrivial homology $H_1(M)\not=0$ the rotation class $\lambda(v)\in H_1(M)$ of a vector field $v$
is a regular invariant of the adjoint action of the identity connected component of the  group $\text{SDiff}(M)$.
\end{corollary}

The rotation class $\lambda$ defines a natural projection $\lambda: \mathfrak X^k\to H_1(M)$, where $\mathfrak X^k_{ex}=\lambda^{-1}(0)$. So intuitively, the helicity (well-defined on $\mathfrak X^k_{ex}$) and the rotation class
together constitute the set of regular invariants of the adjoint action. However for fields with nontrivial rotation class only relative helicity (of one field with respect to another) is well defined, see \cite{arkh}  for details.

This paper is organized as follows. In Section~\ref{sec:main_result} we introduce the  definition of regular integral invariants and prove the main theorem  modulo Proposition~\ref{P:FI}. This proposition is proved in Section~\ref{S:auxiliar}, thus completing the proof of the main result.
In Section~\ref{sec:pfCor} we prove the corollaries and discuss in more detail the rotation class and the setting of adjoint and coadjoint actions on exact and non-exact divergence-free vector fields. Finally, in Section~\ref{sec:settings} we recall the geometric formulations of ideal hydrodynamics and magnetohydrodynamics, and explain how the previous results extend to the latter setting.
% ($k\geq 4$) is a function of magnetic helicity and cross-helicity. We finish this paper with Appendix~\ref{sec:settings} where we review some standard geometric constructions that are used in previous sections, including the coadjoint action of the group of volume-preserving diffeomorphisms.

\bigskip

\noindent {\bf Acknowledgments.} B.K. is grateful to A. Izosimov for fruitful discussions.
B.K. was partially supported by an NSERC research grant.  D.P.-S. was supported by the grants MTM2016-76702-P (MINECO/FEDER) and Europa Excelencia EUR2019-103821 (MCIU), and partially supported by the ICMAT--Severo Ochoa grant SEV-2015-0554. A part of this work was done while C.Y. was visiting
the Instituto de Ciencias Matem\'{a}ticas (ICMAT) in Spain. C.Y. is grateful to the ICMAT for its support and kind hospitality.

%%%%%%%%%%%%%%%%%%

\section{Regular integral invariants and proof of the main theorem}\label{sec:main_result}

Consider the space $\mathfrak X_{ex}^k$, $k\in\{1,2,\cdots,\infty\}$ of \emph{exact divergence-free vector fields} on
a three-dimensional Riemannian manifold $M$. (The notation $ \mathfrak X_{ex} \equiv \mathfrak X_{ex}^\infty$ stands for
$k=\infty$.) Recall that a divergence-free field $w$ is exact if it is the curl of another vector field, or, equivalently, if $i_wd\mu$ is an exact $2$-form on $M$.
%The $\rm curl$ operator on a Riemannian manifold allows one to identify the space $\mathfrak X_{ex}$ and its dual $\mathfrak X_{ex}^*$, see Appendix~\ref{sec:settings} for details.
The reason to consider the space $\mathfrak X_{ex}^k$ is that the curl operator  on exact fields has a well-defined inverse ${\rm curl}^{-1}: \mathfrak X_{ex}^k\to \mathfrak X_{ex}^k$. In this context, one can define the helicity of an exact field as follows:

\begin{definition}\label{def:helicity}
The {\it helicity} is the following quadratic form on $w\in\mathfrak X_{ex}^k$:
$$
\mathcal H(w):=\int_M w\cdot {\rm curl}^{-1} w\,d\mu\,.
$$
\end{definition}

It is well known~\cite{arkh} that the helicity is invariant under the action of the group $\text{SDiff}(M)$ of (smooth, orientation-preserving) volume-preserving diffeomorphisms of the manifold $M$, i.e. the helicity is a Casimir function of the group
$\text{SDiff}(M)$.
For now we need its infinitesimal invariance with respect to the action by a Lie derivative (or a Lie bracket) $L_v: \mathfrak X_{ex} \to \mathfrak X_{ex}$, where $ w\mapsto [w,v]$ for any divergence-free vector field $v$. The fact that $\mathfrak X_{ex}$ is dense in $\mathfrak X_{ex}^k$ allows us to consider the $L_v$ action by smooth fields $v$ on $\mathfrak X_{ex}^k$; we shall use this property in what follows without further mention.

%As explained in Appendix~\ref{sec:settings}, the coadjoint action ${\rm ad}^*_{v}$ on the exact divergence-free vector field $w$, expressed in terms of vector fields, is given by $$ {\rm ad}^*_{v}\,w=[w,v]\,,$$
%where $[\cdot,\cdot]$ denotes the Lie bracket of two vector fields. Obviously, the ${\rm ad}^*$ action on the space of $C^\infty$ exact fields $\mathfrak X_{ex}$ naturally extends to $\mathfrak X_{ex}^k$.
%Let us now introduce the notion of \emph{regular integral invariant} (see also~\cite{EnPeTo}):

\begin{definition}[\cite{EnPeTo}]\label{def:regInt1}
Let $\mathcal F : \mathfrak X_{ex}^k\rightarrow\mathbb R$ be a $C^1$ functional. We say that $\mathcal F $ is
a \emph{regular integral invariant} if:
\begin{enumerate}
\item It is invariant under the action of the volume-preserving diffeomorphism group $\text{SDiff}(M)$, i.e., $\mathcal F (w)= \mathcal F(\Phi_*w)$ for any $\Phi\in \text{SDiff}(M)$.
\item At any point $w\in\mathfrak X_{ex}^k$, the (Fr\'{e}chet) derivative of $\mathcal F $ is an integral operator with
continuous kernel, that is,
\begin{equation}\label{eq:regInt11}
(D\mathcal F )_{w}(u)=\int_M K(w)\cdot\,u\,d\mu,
\end{equation}
for any $u\in\mathfrak X_{ex}^k$, where $ K: \mathfrak X_{ex}^k\rightarrow\mathfrak X_{ex}^k$ is a continuous map (with respect to the Whitney $C^k$ topology). Here the dot stands for the Riemannian inner product of two vector fields on $M$.
\end{enumerate}
\end{definition}
\begin{remark}
The regular integral invariant $\mathcal F :\mathfrak X_{ex}^k\rightarrow \mathbb R$ is $C^1$ (or continuously differentiable) in the Fr\'{e}chet sense, i.e., the Fr\'{e}chet derivative $D\mathcal F$ is a continuous map. Depending on how we identify the dual space $\mathfrak X_{ex}^*$ with $\mathfrak X_{ex}$, we obtain different regularities for the derivative $D\mathcal F$. Indeed, with the identification used in the context of the coadjoint action, introduced in Section~\ref{SC:main}, cf. Equation~\eqref{eq:pair}, we have that $(D\mathcal F )_{w}={\rm curl}\; K(w)$, so $D\mathcal F :\mathfrak X_{ex}^k\rightarrow\mathfrak X_{ex}^{k-1}$. On the other hand, taking the pair between
two exact fields to be the standard $L^2$ product, we have that $(D\mathcal F )_{w}=K(w)$, and hence $D\mathcal F :\mathfrak X_{ex}^k\rightarrow\mathfrak X_{ex}^{k}$. However, this distinction will not be important in what follows because we shall not work with $D\mathcal F$, but directly with the kernel~$K$.
\end{remark}
It is easy to check that the helicity $\mathcal H$ is a regular integral invariant (for any $k\geq 1$) with kernel
\begin{equation}
K(w)=2\,{\rm curl}^{-1} w\,.
\end{equation}

Note that in the case that the manifold $M$ has trivial homology $H_1(M)=0$, any divergence-free vector field is exact. Accordingly, in this case regular integral invariants on the space of smooth fields $\mathfrak X_{ex}$ and regular Casimir functions of the Lie group ${\rm SDiff}(M)$ are the same (cf. Section~\ref{sec:pfCor} for a proof).

\bigskip

\noindent {\bf Proof of the main theorem}.
The proof of Theorem~\ref{T:main} follows the strategy of~\cite{EnPeTo}, while the implementation turns out to be quite different in view of the different regularity of the vector fields.

\begin{proposition}
For a regular  invariant $\mathcal F $ with kernel $K$ on the space $\mathfrak X_{ex}^k$, $k\geq 1$ the kernel satisfies the following property: for any $w\in\mathfrak X_{ex}^k$  there exists a $C^k$ function $J_w:M\to\mathbb R$ (depending on $w$)
such that
\begin{equation}\label{eq:first_Integral1}
{\rm curl}\,K(w)\times w=\nabla J_w\,.
\end{equation}
\end{proposition}

\begin{proof}
By differentiating $\mathcal F ((\phi_t)_*(w))=\mathcal F(w)$ at $t=0$ for a family  of diffeomorphisms $\phi_t$  we obtain after certain transformations
$$
0=\frac {d}{dt}\mathcal F((\phi_t)_*(w))|_{t=0}=\int_M K(w)\cdot [w,v]\,d\mu=-\int_M v\cdot ({\rm curl}\,K(w)\times w)\,d\mu
$$
for any  divergence-free velocity field $v:={d\phi_t}/{dt}(0)\in\mathfrak X(M)$. (One uses the relation
$[w,v]={\rm curl}\,(v\times w)$ for divergence-free vector fields.)
Due to arbitrariness of $v$ this implies that the vector field ${\rm curl}\,K(w)\times w$ is $L^2$-orthogonal to all divergence-free vector fields on $M$, and hence it is a gradient field.
\end{proof}

It turns out that for sufficiently smooth $w$ one can prove a stronger result:

\begin{proposition}\label{P:FI}
For the kernel $K$ of a regular invariant $\mathcal F $ on the space $\mathfrak X_{ex}^k$, $k\geq 4$, for any $w\in\mathfrak X_{ex}^k$
one has
$$
{\rm curl}\,K(w)\times w=0\,.
$$
\end{proposition}

We shall prove this key proposition in the next section. For now we assume it is true and continue with the proof of Theorem~\ref{T:main}.

The equation ${\rm curl}\,K(w)\times w=0$ implies that there exists a function $g\in C^{k-1}(M\backslash w^{-1}(0))$ defined in the complement of the zero set of $w$, such that
$${\rm curl}\,K(w)=g\;w,\;\;\text{with}\;\; g:=\frac{w\cdot {\rm curl}\,K(w)}{|w|^2}\,.$$

Consider the set $\mathcal R$ of exact divergence-free vector fields $w\in\mathfrak X_{ex}^k $, whose zeros are hyperbolic (and hence isolated, so there are finitely many of them). We will call vector fields in $\mathcal R$ hyperbolic.
A standard transversality argument (see e.g.~\cite[Chapter 2.3]{PM}) implies that $\mathcal R$ is an open and dense set of $\mathfrak X_{ex}^k$. Take a vector field $w_0\in\mathcal R\subset \mathfrak X_{ex}^k$.

\begin{proposition}\label{prop:step3}
The  function $g_0(x)$ corresponding to a hyperbolic vector field $w_0\in\mathcal R\subset \mathfrak X_{ex}^k$
 can be extended to the whole manifold $M$ as a $C^{k-2}$ function, i.e. for any $w_0\in\mathcal R$
 there exists a function
 $g_0\in C^{k-1}(M\backslash w_0^{-1}(0))\bigcap C^{k-2}(M)$  such that ${\rm curl}\,K(w_0)=g_0(x)\,w_0$.
\end{proposition}

Indeed, the equation ${\rm curl}\,K(w_0)\times w_0=0$ can be equivalently written as $i_{{\rm curl}\,K(w_0)} i_{w_0}d\mu=0$.
This implies collinearity of the vector fields ${\rm curl}\,K(w_0)$ and ${w_0}$, and hence
the existence of the required $C^{k-1}$ function $g_0$ away from the zero set of $w_0$. Furthermore,
the set $w_0^{-1}(0)$ consists of finitely many points, which are hyperbolic zeros of the field $w_0$. Then Proposition~\ref{prop:step3} immediately follows from the following version of  Hadamard's division lemma, applied in a neighborhood of each zero of $w_0$, which allows one to extend the function $g_0$ to a $C^{k-2}$-function on the entire $M$. The statement
of this division lemma is local and can be formulated in local coordinates $(x_i)$
in a neighborhood of the origin of $\mathbb R^n$.

\begin{lemma}\label{lem:division}
Let $v=\sum \alpha_i\partial_{x_i}$ and $w=\sum\beta_i\partial_{x_i}$ be two $C^l$ vector fields defined in a neighborhood $V$ of $0\in\mathbb R^n$, and $d\mu$ a volume form on $\mathbb R^n$. Assume that $0$ is a hyperbolic zero of $w$. Then, if $i_vi_wd\mu=0$ in $V$, there exists a function $h\in C^{l-1}(V)$, such that
$\alpha_i=h(x)\beta_i$, for all $i=1,...,n$.
\end{lemma}

\begin{proof}
This is a finite-smoothness version of the theorem in~\cite{Mou76} on zeros of differential forms, which is also a version of
Hadamard's lemma. Due to the assumption $i_vi_wd\mu=0$, these two fields are collinear outside of the origin and
 there exists a function $h\in C^l(V\setminus \{0\})$, such that $\alpha_i=h(x)\beta_i$, for all $i=1,...,n$ and $x\not=0$.

Since the field $w$ has a hyperbolic singularity at $0$, one can use the functions $\beta_i$ as local coordinates $\tilde x_i$ near the origin; any function in the new coordinate system will be denoted with a ``tilde''. Since the collinearity assumption implies that $v$ vanishes at the origin (otherwise the zero of $w$ cannot be hyperbolic, this follows from a simple flow box argument), then $\tilde \alpha_i =0$ for $\tilde x_i=0$. One so has $\tilde \alpha_i =\tilde h_i(\tilde x) \tilde x_i$ by the classical Hadamard lemma, while the collinearity condition implies that $\tilde h_i=\tilde h_j=\tilde h$ in $V\setminus \{0\}$.
The smoothness of $h$ in the whole of $V$
is less by 1 than that of $\tilde \alpha_i$, which is shown by a standard argument:
in a convex neighborhood $\tilde V\subset V$ of the origin (in new coordinates) we can express
the components $\tilde \alpha_i =\tilde h(\tilde x) \tilde x_i$ of the vector field $v$ as follows:
$$
\widetilde{\alpha}_i(\tilde x)=\int_0^1 \frac{d}{dt}\widetilde{\alpha}_i(\tilde x_1,\cdots,t\tilde x_i,\cdots,\tilde x_n)\,dt
=\left(\int_0^1 \partial_i\widetilde{\alpha}_i(\tilde x_1,\cdots,t\tilde x_i,\cdots,\tilde x_n)\,dt\right)\tilde x_i
= \tilde {h}(\tilde x )\, \tilde x_i\,,
$$
where $\tilde{h}(\tilde x ):= \int_0^1 \partial_i\widetilde{\alpha}_i(\tilde x_1,\cdots,t\tilde x_i,\cdots,\tilde x_n)\,dt$. Since
$\tilde \alpha_i\in C^{l}(\tilde V)$, one obtains that $\tilde h\in C^{l-1}(\tilde V)$ and therefore $h=\tilde h \circ \beta\in C^{l-1}( V)$, where $\beta=(\beta_1,\beta_2,\cdots,\beta_n)$.
\end{proof}

Now, according to Proposition~\ref{prop:step3}, we have the collinearity of the fields ${\rm curl}\,K(w_0)$ and $w_0$. This can be used to show that the derivatives of the functional $\mathcal F $ and the helicity are proportional:

\begin{proposition}\label{prop:prop}
For a regular invariant functional $\mathcal F $ on $\mathfrak X^k_{ex}$, there exists a continuous functional $\mathcal C:\mathfrak X^k_{ex}\backslash\{0\}\to \mathbb R$ such that
$$(D\mathcal F )_w=\mathcal C(w)\,(D\mathcal H)_w\,,$$
for all $w\in\mathfrak X^k_{ex}\backslash\{0\}$.
\end{proposition}

\begin{proof}
Since $0={\rm div}\,{\rm curl}\,K(w_0)=\nabla g_0\cdot w_0$, we have that the function $g_0$ is a $C^{k-2}$ first integral of the vector field $w_0\in\mathcal R$ (more precisely, $g_0\in C^{k-1}(M\backslash w_0^{-1}(0))\bigcap C^{k-2}(M)$). The zero set of $w_0$ consists of finitely many points that we denote by $\{p_1,p_2,\dots,p_N\}\subset M$. Take the values $g_0(p_i)=c_i$, and a point $p\in M$ with $g_0(p)\neq c_i$ for all $i$. Since $g_0$ is $C^{k-1}$ in the complement of the finite set $\{p_1,p_2,\dots,p_N\}$, Sard's theorem implies that one can safely assume that $g_0(p)$ is a regular value. It then follows from Lemma~\ref{lem:reg_levelset} (see Section~\ref{S:auxiliar}), Thom's isotopy theorem and the compactness of the manifold that there is a domain $U$ containing $p$ which is trivially fibred by components of the level sets of $g_0$, that are invariant tori of $w_0$.

Now using Lemma~\ref{lem:torus} in Section~\ref{S:auxiliar}, and arguing exactly as in the proof of Proposition~\ref{P:FI}, we conclude that $g_0$ must be a constant for any $w_0\in\mathcal R$, i.e., one has that ${\rm curl}\,K(w_0)=2\,\mathcal C({w_0})\,w_0$ for all $w_0\in\mathcal R$, where $\mathcal C({w_0})$ is a constant on $M$ that only depends on the field $w_0$. The continuity of the kernel $K$ (which implies the continuity of the function $g$ on $\mathfrak X^{k}_{ex}$), and the fact that $g$ is a constant on $M$ (that depends on $w$) for a dense subset $\mathcal R$ of $\mathfrak X^{k}_{ex}$, imply that $g$ must also be a constant (depending on $w$) for all $w\in\mathfrak X^k_{ex}\backslash\{0\}$.

Summarizing, we have proved that there exists a continuous functional $C:\mathfrak X^k_{ex}\backslash\{0\}\to \mathbb R$ such that the derivative of $\mathcal F $ reads as
$$
(D\mathcal F )_w(u)=\int_MK(w)\cdot u\,d\mu=\int_M 2\,\mathcal C(w)\,\text{curl}^{-1}w\cdot u\,d\mu
$$
for any $u\in\mathfrak X_{ex}^k$ and $w\in\mathfrak X^k_{ex}\backslash\{0\}$. Noticing that the derivative of the helicity is given by $(D\mathcal H)_w(u)=2\int_M{\rm curl}^{-1}w\cdot u\,d\mu$, we obtain that
$$(D\mathcal F )_w=\mathcal C(w)\,(D\mathcal H)_w\,,$$
for all $w\in\mathfrak X^k_{ex}\backslash\{0\}$.
\end{proof}

Now, since the differentials of the functional $\mathcal F $ and the helicity $\mathcal H$ are proportional by Proposition~\ref{prop:prop}
at each point of the space $\mathfrak X^k_{ex}\backslash\{0\}$, the main theorem will follow from the path-connectedness of the level sets of $\mathcal H$, which is established in the following proposition. This result was first proved in~\cite{EnPeTo}. For the sake of completeness, here we give a simpler and more transparent proof.

\begin{proposition}[\cite{EnPeTo}]\label{lem:pathconn}
The level sets of the  helicity $\mathcal H^{-1}(c)$ are path-connected subsets of $\mathfrak X^k_{ex}\setminus \{0\}$.
\end{proposition}

\begin{proof}
Let $w_1$ and $w_2$ be two exact divergence-free vector fields with the same helicity, which we first assume to be nonzero, i.e.
$$
\mathcal H(w_1)=\mathcal H(w_2)=c\not=0\,.
$$
In order to prove the connectedness of the $c$-level,
we introduce an auxiliary vector field $\bar\beta\in \mathfrak X^k_{ex}$ with the same helicity $c$, which can be connected with
each $w_0$ and $w_1$.  The only ingredient we need in the proof is the property that the curl operator acting on the space
$\mathfrak X^k_{ex}$ of exact fields has infinitely many positive and negative eigenvalues, which implies that the positive and negative subspaces of the helicity quadratic form $\mathcal H(w)$ on $\mathfrak X^k_{ex}$ are infinite-dimensional.

Assume that $c>0$. Consider the subspace $\mathcal S\subset \mathfrak X^k_{ex}$ of vector fields orthogonal to
the two vector fields $w_1,w_2$ with respect to the helicity quadratic form $\mathcal H$, that is
$$
\mathcal S:=\{u\in \mathfrak X^k_{ex}(M)~:~ \int_M u\cdot \text{curl}^{-1} w_1\,d\mu=\int_M u\cdot \text{curl}^{-1} w_2\,d\mu=0\}\,.
$$
This space has codimension $\le 2$ and hence the restriction of the helicity $\mathcal H$ to this subspace is still sign-indefinite. Hence, one can choose a vector field $\beta\in \mathcal S$ such that $\mathcal H(\beta)=1$ (for $c<0$
one needs to choose  $\mathcal H(\beta)=-1$).

Now define a family of vector fields $w_t:= tw_1+f(t)\beta$ for $t\in [0,1]$, and choose an appropriate function $f(t)$
so that the condition $\mathcal H(w_t)=c$ holds for all $t$. Namely,
$$
\mathcal H(w_t)=\mathcal H\Big(tw_1+f(t)\beta\Big)
=t^2\mathcal H(w_1)+f(t)^2\mathcal H(\beta)= t^2 c+f(t)^2\,,
$$
where we have used that $\mathcal H(w_1)=c$, $\int_M \beta\cdot \text{curl}^{-1} w_1d\mu=\int_M w_1\cdot \text{curl}^{-1}\beta\,d\mu=0$ and $\mathcal H(\beta)=1$.
Then, taking $f(t):=\sqrt{(1-t^2)c}$ for $t\in [0,1]$ we obtain a continuous family $w_t$ of fields in $\mathfrak X^k_{ex}$
that have constant helicity $c$ and connect $\bar\beta:=\sqrt c\beta$ and $w_1$. In the same way one can connect the fields  $w_2$ and $\bar\beta$  for $t\in [-1,0]$. The connecting path can be smoothened out by adjusting this construction. This proves the path-connectedness of the levels sets of the helicity for $c>0$; a similar construction works for $c<0$.

The case of $c=0$ is analogous; take a (non-zero) field $\beta$ with $\mathcal H(\beta)=0$ and the family  $w_t=tw_1 +(1-t)\beta$.
Then $\mathcal H(w_t)=0$ and this family  provides the connectedness of the zero level set of the helicity, as required.
\end{proof}

We are now ready to finish the proof of Theorem~\ref{T:main}.

\begin{proof}
First, we recall that we showed in Proposition~\ref{prop:prop} that the derivative of the functional $\mathcal F $ satisfies
\begin{equation}\label{eq_conn}
(D\mathcal F )_w=\mathcal C(w)\,(D \mathcal H)_w
\end{equation}
for all $w\in\mathfrak X^k_{ex}\backslash\{0\}$ and some continuous functional $\mathcal C(w)$. Let us take any two exact vector fields on the level set $\{\mathcal H=c\}$ of the helicity, and connect them by
a continuous path $w_t\in \mathfrak  X^k_{ex}\backslash\{0\}$ (cf. Proposition~\ref{lem:pathconn}) with constant helicity, i.e. $\mathcal H(w_t)=c$. The derivative of $\mathcal F $ along this path is given by
$$
(D\mathcal F )_{w_t}(\dot w_t)
=\mathcal C({w_t})(D\mathcal H)_{w_t}(\dot w_t)=0\,,
$$
where $\dot w_t$ is the tangent vector along the path for almost all $t$ (actually, one can safely assume that the path is smooth in $t$). This implies that $\mathcal F $ is also constant on each level set of $\mathcal H$. Accordingly, there exists a function $f : \mathbb R\rightarrow\mathbb R$
which assigns a value of $\mathcal F $ to each value of the helicity,
i.e., $\mathcal F (w) =f(\mathcal H(w))$ for all $w\in \mathfrak X^k_{ex}\backslash\{0\}$.
To include the case $w=0$, we observe that $\mathcal F $ is
constant on the level set $\mathcal H^{-1}(0)\backslash\{0\}$, so the continuity of the functional $\mathcal F $ in $\mathfrak  X^k_{ex}$ implies that it takes the same constant
value on the whole level set $\mathcal H^{-1}(0)$, and hence the property $\mathcal F (w) =f(\mathcal H(w))$ holds for all $w\in \mathfrak X^k_{ex}$. Additionally, $f$ is of class $C^1$ since $\mathcal F $ itself is a $C^1$ functional. This completes the proof of Theorem~\ref{T:main}.
\end{proof}

\section{Properties of the invariant kernel}\label{S:auxiliar}

The goal of this section is to prove Proposition~\ref{P:FI}. First, notice that $0=w\cdot({\rm curl}\,K(w)\times w)=w\cdot\nabla J$, and hence $J$ is a first integral of $w$. By the continuity of the kernel $K$, Proposition~\ref{P:FI} follows if we show that $J$ is a constant (a trivial first integral) on the manifold $M$ for a residual (and hence dense) set of vector fields $w\in \mathfrak X_{ex}^k$.

To this end, consider the set $\mathcal R$ of exact divergence-free vector fields $w\in\mathfrak X_{ex}^k $, whose zeros are hyperbolic (and hence isolated). Each vector field $w\in \mathcal R$ has finitely many zeros. This set was already introduced after the statement of Proposition~\ref{P:FI}, where it was pointed out that $\mathcal R$ is an open and dense set of $\mathfrak X_{ex}^k$. The proof of Proposition~\ref{P:FI} makes use of the following instrumental lemma.

\begin{lemma}\label{lem:reg_levelset}
If $w\in\mathcal R$ and $P$ is a first integral of $w$ of class $C^{k-1}$, $k\geq 4$, then any component of a regular level set of $P$ is diffeomorphic to the torus $\mathbb T^2$.
\end{lemma}
\begin{proof}
By the assumption $k\geq 4$, Sard's theorem implies that most of the values of the first integral $P$ are regular. Let $c\in\mathbb R$ be a regular value, then a connected component $\Sigma_c$ of the level set $P^{-1}(c)$ is an oriented compact surface in $M$. Assume that $\Sigma_c$ is diffeomorphic to a surface of genus $\nu\neq 1$, i.e., the surface is not a torus. The Poincar\'{e}-Hopf theorem easily implies that there exists a point $p_c\in \Sigma_c$ such that $w(p_c)=0$ (and the index of $w$ at $p_c$ is nonzero). It follows from the compactness of the manifold that any value $c'\in(c-\delta,c+\delta)$, $\delta$ small enough, is regular as well. Moreover, Thom's isotopy theorem~\cite{AR67} implies that the level set $P^{-1}(c')$ is isotopic to $P^{-1}(c)$, and therefore there is a neighborhood of $\Sigma_{c}$ that is saturated by components $\Sigma_{c'}$ of the level sets of $P$ isotopic to $\Sigma_c$. Since on each surface $\Sigma_{c'}$, $c'\in(c-\delta,c+\delta)$, the field $w$ vanishes at a point $p_{c'}$, we conclude that $w$ has infinitely many zeros on $M$ (in fact,  a continuous arc of zeros). This contradicts the assumption that $w\in \mathcal R$, and therefore the genus of any component of a regular level set must be $1$, i.e., it is diffeomorphic to a torus.
\end{proof}

Take now a vector field $w_0\in\mathcal R\subset \mathfrak X_{ex}^k$, and assume that the corresponding $C^k$ function $J_0$ defined by Equation~\eqref{eq:first_Integral1} is not a constant. Let $c\in\mathbb R$ be a regular value of $J_0$ (which exists by Sard's theorem). By Lemma~\ref{lem:reg_levelset}, all the components of $J_0^{-1}(c)$ are diffeomorphic to a torus; in fact, there exists a domain $U\subset M$ which is trivially
fibred by toroidal components of the level sets of $J_0$ (which are invariant tori of $w_0$), so in particular $U\cong \mathbb T^2\times (0,1)$. The following lemma shows that an arbitrarily small perturbation can destroy this trivial fibration of invariant tori.

\begin{lemma}\label{lem:torus}
For any $\epsilon>0$, there exists a vector field $w\in \mathfrak X_{ex}^k$ such that $\|w-w_0\|_{C^k}<\epsilon$, $w=w_0$ on $M\backslash U$, and $w$ does not admit a trivial fibration of invariant tori on $U$.
\end{lemma}
\begin{proof}
The main idea of the proof is the following two-step procedure. First, we show that arbitrarily close to $w_0$ there is a field $w_1$ with a family of invariant tori, such that the winding number of trajectories on tori changes nondegenerately across the family.
After that, by choosing an invariant torus with a rational winding number $p/q$ (which exists due to this nondegeneracy), we show that a generic perturbation of such a field $w_1$ leads to a field whose Poincar\'e return map has at least $q$ hyperbolic saddles, and hence does not admit a fibration into invariant surfaces.

\smallskip

In more detail, by assumption, the domain $U \subset M$ is diffeomorphic to $\mathbb T^2\times (0,1)$ and each torus $\mathbb T^2\times \{r\}$, $r\in(0,1)$, is invariant under the flow of $w_0$.
Introduce the field $\xi:={\nabla J_0}/{(\nabla J_0)^2}$.
It is easy to check that the field $w_0|_{T_r}$ preserves the $C^{k-1}$ area form
\begin{equation}\label{eq:area}
d\mu_2:=i_{\xi}\,d\mu\Big|_{T_r}\,,
\end{equation}
since $T_r:=\mathbb T^2\times \{r\}$ corresponds to a regular level set of $J_0|_U$.

Let us introduce coordinates $(\theta_1,\theta_2,r)$ parameterizing $U$, where $(\theta_1,\theta_2)\in\mathbb T^2=(\Bbb R/2\pi\Bbb Z)^2$ and $r\in(0,1)$. Since $w_0$ preserves the area form $d\mu_2$ on each torus $T_r$, Sternberg's theorem~\cite{Stern} implies that the angular coordinates can be taken so that the vector field $w_0|_U$ in these coordinates reads as:
$$w_0=F(\theta_1,\theta_2,r)\Big(\Theta_1(r)\partial_{\theta_1}+\Theta_2(r)\partial_{\theta_2}\Big)\,,$$
where $\Theta_1(r),\,\Theta_2(r)$ are $C^{k}$ functions, and $\mathcal F $ is a nonvanishing $C^{k-1}$ function.

Let us now perturb $w_0$ to obtain a vector field $w_1\in \mathfrak X_{ex}^k$ such that $w_1=w_0$ on $M\backslash U$, $\|w_1-w_0\|_{C^k}<\delta$ for some $\delta$ that will be fixed later, and on $U$ there are coordinates (that we still denote by $(\theta_1,\theta_2,r)$) such that
$$
w_1=\widetilde F(\theta_1,\theta_2,r)\Big(\widetilde\Theta_1(r)\partial_{\theta_1}+\widetilde\Theta_2(r)\partial_{\theta_2}\Big)\,,
$$
where the $C^{k-1}$ function $\widetilde F$ does not vanish on $U$, and the $C^{k}$ functions $\widetilde\Theta_1,\widetilde\Theta_2$ satisfy on $(0,1)$ that $\widetilde\Theta_2$ has finitely many zeros and
the ratio $\Phi(r):={\widetilde \Theta_1}/{\widetilde \Theta_2}$ is not identically constant. To prove that such a perturbation exists, it is enough to take a $C^k$ vector field $Q$ that is zero on $M\backslash U$, and on $U$ is given in $(\theta_1,\theta_2,r)$-coordinates by
$$
Q|_U:=\frac{1}{h(\theta_1,\theta_2,r)}\Big(f_1(r)\partial_{\theta_1}+f_2(r)\partial_{\theta_2}\Big)\,,
$$
where $h$ is the function that appears in the volume form when written in local coordinates, i.e. $d\mu=h\,d\theta_1\wedge d\theta_2 \wedge dr$, and $f_1,f_2$ are $C^k$ functions that are $0$ in a neighborhood of $r=0$ and $r=1$, and $\|f_1\|_{C^k}<C\delta$, $\|f_2\|_{C^k}<C\delta$. It is straightforward to check that the field $Q$ is divergence-free with respect to the volume $d\mu$, and it is exact because $i_Qd\mu=f_1(r)d\theta_2\wedge dr-f_2(r)d\theta_1\wedge dr$ is an exact $2$-form on $U$, while $i_Qd\mu=0$ on $M\backslash U$, so the field $Q$ belongs to the space $\mathfrak X_{ex}^k$. Defining $w_1:=w_0+Q$, a generic choice of the functions $f_1$ and $f_2$ gives the desired properties for $w_1$.

The properties above allow us to take an interval $(a,b)\subset (0,1)$ where $\widetilde \Theta_2$ does not vanish and $\Big|\frac{d}{dr}\Phi(r)\Big|>0$. The Poincar\'{e} return map of the vector field $w_1$ at the cross section $\mathbb S^1\times\{\theta_2=0\}\times (a,b)\subset U$ is well defined and given by
$$
\Pi(\theta_1,r)=\Big(\theta_1+2\pi\Phi(r),\,r\Big)\,.
$$
This is a $C^{k}$ diffeomorphism of the annulus $\mathbb S^1\times(a,b)$ that preserves an area form $d\sigma$ (which in these coordinates takes the form $f(r)d\theta_1\wedge dr$ for some positive function $f$, but we will not use this specific form in what follows). Moreover, it satisfies the twist condition:
$$
\Big|\frac{d}{dr}\Big( 2\pi\Phi(r)\Big)\Big|>0\,,
$$
where $\Phi(r)=\widetilde \Theta_1/{\widetilde \Theta_2}$.

Let $c\in(a,b)$ be a value of the $r$-coordinate such that $2\pi \Phi(c)= p/q$ for some coprime natural numbers $p$ and $q$. The circle $\{r=c\}$ is then formed by fixed points of the iterated map $\Pi^q$. The twist assumption implies that on the invariant circles $\{r<c\}$ and $\{r>c\}$ the map $\Pi^q$ rotates in opposite directions. Take now a $C^k$ diffeomorphism of the annulus $\Pi_\delta: \mathbb S^1\times(a,b)\rightarrow\mathbb S^1\times(a,b)$ which is $C\delta$-close to $\Pi$, i.e. $\|\Pi_\delta-\Pi\|_{C^{k}}<C\delta$, preserves the area form $d\sigma$, and $\Pi_\delta=\Pi$ in a neighborhood of $\mathbb S^1\times\{a\}$ and $\mathbb S^1\times\{b\}$. The Poincar\'{e}-Birkhoff theorem~\cite[Section 4.8]{GuHo} shows that a generic perturbation $\Pi_\delta^q$ has at least $q$ fixed points which are hyperbolic saddles and are $C\delta$ close to the circle $\{r=c\}$.
These fixed points correspond to hyperbolic $q$-periodic points of the map $\Pi_\delta$ that bifurcate from the resonant circle $\{r=c\}$ as $\delta\to 0$.

To conclude the proof of the lemma, we now take the suspension of the diffeomorphism $\Pi_\delta$ along the $\theta_2$-direction to obtain a $C^k$ vector field $w_2$ in the domain $V:=\mathbb T^2\times(a,b)\subset U$ whose Poincar\'{e} return map at the cross section $\mathbb S^1\times\{\theta_2=0\}\times (a,b)$ is precisely $\Pi_\delta$. Ref.~\cite{Tres} shows that this suspension can be taken such that $w_2$ is divergence-free with respect to the volume $d\mu$, $w_2=w_1$ in a neighborhood of $\partial V$, and $\|w_2-w_1\|_{C^k}<C\delta$. We can then define the vector field $w$ on $M$ as
$$
w:=\left\{\begin{array}{l}
w_2\;\;\;\; {\rm in}\;\; V,\\
w_1\;\;\;\;{\rm in}\;\; M\backslash V\,.
\end{array}\right.
$$
It is obvious that $w$ is $C^k$, divergence-free, and $\|w-w_0\|_{C^k}<\epsilon$ taking $\delta=\epsilon/C$.  Moreover, we claim that it is exact. Indeed, since $w_1$ is exact, it is enough to show that the divergence-free field $R:=w-w_1$ is also exact. This follows from the fact that $R$ is supported on $V=\mathbb T^2\times (a,b)$, which implies that $\int_{\mathbb T^2\times \{r\}}i_Rd\mu=0$ for all $r\in (a,b)$; since any surface $\mathbb T^2\times \{\cdot\}$ is a generator of the second homology group of $V$, we infer from De Rham's theorem that the closed $2$-form $i_Rd\mu$ is exact. Finally, notice that the properties of the map $\Pi_\delta$ proved above imply that $w$ has (at least) a hyperbolic periodic orbit in the domain $V\subset U\subset M$, thus following that $w$ cannot admit a trivial fibration by invariant tori on $U$.
\end{proof}

We are now ready to prove Proposition~\ref{P:FI}. Take the vector field $w_0\in\mathcal R\subset \mathfrak X_{ex}^k$ introduced before the statement of Lemma~\ref{lem:torus} that satisfies Equation~\eqref{eq:first_Integral1} with some non-constant $C^k$ function $J_0$. As discussed above, there is a domain $U\subset M$ which is trivially fibred by toroidal components of the level sets of $J_0$ (which are invariant tori of $w_0$). For any $w\in \mathfrak X_{ex}^k$ as in the statement of Lemma~\ref{lem:torus}, which is $\epsilon$-close to $w_0$, that is $\|w-w_0\|_{C^k}<\epsilon$, the function $J$ defined by Equation~\eqref{eq:first_Integral1} satisfies the estimate:
$$
\|\nabla(J-J_0)\|_{C^{k-1}}=\|{\rm curl}\,K(w)\times w-{\rm curl}\,K(w_0)\times w_0\|_{C^{k-1}}<C\|w-w_0\|_{C^k}\,,
$$
where the constant $C$ depends only on $w_0$. Here we have used that the kernel $K:\mathfrak X_{ex}^k\to \mathfrak X_{ex}^k$ is continuous with respect to the $C^k$ Whitney topology. Since the function $J$ is defined up to a constant, the Poincar\'e inequality and this estimate imply that
$$
\|J-J_0\|_{C^{k}}<C\epsilon\,.
$$
It then follows from Thom's isotopy theorem~\cite{AR67} that the function $J$ defines a trivial fibration by invariant tori of $w$ in the domain $U$. However, this contradicts Lemma~\ref{lem:torus}, which ensures that $U$ is not trivially fibred by invariant tori of $w$.

We conclude that for any $w_0\in\mathcal R\subset \mathfrak X_{ex}^k$, the corresponding first integral $J_0$ is a constant on $M$. The continuity properties of the kernel $K$ and the density of the set $\mathcal R$ readily imply that ${\rm curl}\,K(w)\times w=0$, for any $w\in\mathfrak X_{ex}^k$, as we desired to prove.

\section{Adjoint and coadjoint invariants of volume-preserving \\diffeomorphisms}\label{sec:pfCor}
%Proof of Corollary~\ref{C:main}}\label{sec:pfCor}

In this section we describe the geometry of the adjoint and coadjoint actions of the group of
volume-preserving diffeomorphisms and prove Corollaries \ref{C:main} and \ref{cor:adjoint}.

\subsection{Invariants of the adjoint action on divergence-free vector fields and proof of Corollary~\ref{cor:adjoint}}

The Lie algebra of the group ${\rm SDiff}(M)$ of (smooth) volume-preserving diffeomorphisms
on a manifold $M$ with a volume form $d\mu$
 is the space $\mathfrak X$ of all (smooth) divergence-free vector fields on $M$.
The adjoint action of the Lie algebra on itself is by means of the Lie bracket of vector fields:
given $v\in \mathfrak X$, the operator ${\rm ad}_v: w \mapsto [w,v]$.

The divergence-free condition on a field $w$ is equivalent to the condition that the 2-form $i_w d\mu$ is closed on $M$.
A divergence-free vector field $\xi$ on $M$ is exact if the closed $2$-form $i_{\xi}\,d\mu$ is exact. The space $\mathfrak X_{ex}$
of all exact divergence-free vector fields on $M$ is a Lie subalgebra of the Lie algebra $\mathfrak X$ of all divergence-free vectors.

Note that ${\rm ad}_v$ leaves $\mathfrak X_{ex}$ invariant. Indeed, for $w\in \mathfrak X_{ex}$ and $v\in  \mathfrak X$
one has $i_{w}\,d\mu=d\alpha$ for a 1-form $\alpha$,
then
$$
i_{[w,v]}d\mu =[i_w,L_v]d\mu=i_w(L_vd\mu)-L_v(i_wd\mu)=-L_v d\alpha = -dL_v\alpha\,,
$$
i.e. ${\rm ad}_vw=[w,v]\in \mathfrak X_{ex}$.

Furthermore, for any vector field $\xi\in \mathfrak X$ one can define its rotation class $\lambda(\xi)\in H_1(M, \mathbb R)$ as follows, see~\cite{Arn69}. Algebraically, $\lambda(\xi)$ is given by the cohomology class of the closed 2-form $i_\xi d\mu$ in $H^2(M,\mathbb R)$, which is Poincar\'e isomorphic to $H_1(M,\mathbb R)$. Equivalently, if the set of $1$-forms $\{h_k\}_{k=1}^N$ is a basis of the first cohomology group of $M$, where $N$ is the first Betti
number of $M$, the rotation class $\lambda(\xi)$ is a vector in $\mathbb R^N \cong H_1(M,\mathbb R)$ whose components are given by
$$
(\lambda(\xi))_k=\int_M i_\xi d\mu\wedge h_k\,.
$$

Geometrically, the rotation class is the averaging of asymptotic cycles defined by the trajectories of the field $\xi$: for any $x\in M$ take the trajectory $g^t_\xi(x)$ of the field $\xi$ starting at $x=g^0_\xi(x)$ for
time $t\in [0,T]$, and then close it up by a ``short path" (e.g. a geodesic) between $x$ and $g^T_\xi(x)$. This closed path defines an element
$\lambda(\xi, x, T)\in H_1(M, \mathbb R)$. Then $\lambda(\xi, x):=\lim_{T\to\infty}{ \lambda(\xi, x, T)}/{T}\in H_1(M, \mathbb R)$ is an asymptotic cycle associated with a point $x\in M$. It defines an element of $L^1(M)$ by the Birkhoff ergodic theorem,
and one can prove that $\lambda(\xi)=\int_M \lambda(\xi, x)\,d\mu$, so both the algebraic and the geometric definitions of $\lambda(\xi)$ coincide, see~\cite{Arn69}.

Note that $\mathfrak X_{ex}=\{\xi\in \mathfrak X~:~\lambda(\xi)=0\}$, i.e. exact divergence-free vector fields are exactly those having zero rotation class
on $M$. Moreover, since any two closed $2$-forms diffeomorphic via a diffeomorphism connected with the identity
have the same rotation class,  the group action of the connected component of the identity of ${\rm SDiff}(M)$  (or the  action of
the Lie algebra $\mathfrak X$) preserves fibers of the projection $\lambda: \mathfrak X\to H_1(M, \mathbb R)$.
Helicity $\mathcal H$ is then the only regular invariant on the zero fiber, $\mathfrak X_{ex}=\lambda^{-1}(0)$.
Applying this consideration to vector fields of class $C^k$, $k\geq 4$, one proves Corollary~\ref{cor:adjoint}.

Intuitively, the helicity value $\mathcal H$ and the homology class $ H_1(M, \mathbb R)$
describe the full set of regular invariants. However, for non-zero fibers of the projection $\lambda$, only relative helicity between two different fields is well-defined, cf.~\cite{arkh}.

\subsection{Invariants of the coadjoint action and proof of Corollary~\ref{C:main}}\label{SC:main}

The dual of the Lie algebra $\mathfrak X$ of all (smooth) divergence-free vector fields on $M$
is isomorphic to the space $\mathfrak X^*=\Omega^1(M)/d\Omega^0(M)$ of all (smooth) 1-forms, modulo all exact 1-forms on $M$.
The dual space to $\mathfrak X_{ex}$ is the space $\mathfrak X_{ex}^*=\Omega^1(M)/Z^1(M)$  of all (smooth) 1-forms
modulo all (smooth) closed 1-forms  on $M$.
The natural embedding $\mathfrak X_{ex}\subset \mathfrak X$ corresponds to the natural projection
$\pi: \mathfrak X^*\to  \mathfrak X_{ex}^*$, i.e. projections of cosets, $\pi: \Omega^1(M)/d\Omega^0(M) \to \Omega^1(M)/Z^1(M)$.
The fibers of this projection are finite-dimensional and isomorphic to $Z^1(M)/d\Omega^0(M)=H^1(M, \mathbb R)$.

There is an equivalent and simpler description of the dual space $\mathfrak X_{ex}^*=\Omega^1(M)/Z^1(M)$ as the space of
exact 2-forms on $M$, by taking the differential of 1-forms: $\Omega^1(M)/Z^1(M)\cong d\Omega^1(M)$.
For three-dimensional $M$ with volume form $d\mu$, the space $\mathfrak X_{ex}^*=d\Omega^1(M)$
can also be identified with the space $\mathfrak X_{ex}$ of exact divergence-free vector fields: an exact 2-form $d\alpha$ is associated with an exact field $\xi$ by $i_\xi d\mu=d\alpha$. This is why one can consider both adjoint and coadjoint action on the space of exact vector fields on a three-dimensional manifold.

\begin{remark}
In more detail, to relate the vorticity fields $\xi={\rm curl}\,v$ and the corresponding cosets $[u]$ of 1-forms $u=v^\flat$, we introduce the operator $\sigma: \xi \mapsto [u]$ defined by $u:=({\rm curl}^{-1}\xi)^\flat$, i.e.
$\sigma = \mathbb I\circ {\rm curl}^{-1}$. Here the map $\mathbb I: \mathfrak X\to \mathfrak X^*=\Omega^1(M)/d\Omega^0(M)$
is the inertia operator from the Lie algebra  $\mathfrak X$ of divergence-free vector fields to its dual:
given a vector field $v$ on a Riemannian manifold $M$,  one defines the 1-form $v^\flat$
as the pointwise inner product with  the velocity field $v$,
$v^\flat(w): = (v,w)$ for all $w\in T_xM$.
Note that although both $\mathbb I$ and ${\rm curl}^{-1}$ are metric dependent, the operator $\sigma$ depends on the volume form $d\mu$ only, since $\xi$ is the kernel of $d[u]$, i.e. $i_{\xi}d\mu=d[u]=d\,\sigma(\xi)$.
The map $\sigma$ gives the isomorphism between the spaces $\mathfrak X_{ex}$ and $\mathfrak X_{ex}^*$.

Using this operator $\sigma$, one can identify the elements in $\mathfrak X_{ex}$ and $\mathfrak X_{ex}^*$.
The natural pairing $\langle\cdot,\cdot\rangle$ between $\mathfrak X_{ex}^*$ and $\mathfrak X_{ex}$ becomes
\begin{equation}\label{eq:pair}
\langle \xi,v\rangle=\langle [u],v\rangle=\int_M \text{curl}^{-1}\xi\cdot v\;d\mu\,,
\end{equation}
where $\xi\in\mathfrak X_{ex}$, $v\in\mathfrak X_{ex}$ and $[u]=\sigma(\xi)\in\mathfrak X_{ex}^*$.
\end{remark}

By the identification of spaces $\mathfrak X_{ex}$ and $\mathfrak X_{ex}^*$, one can introduce the coadjoint action on $\mathfrak X_{ex}$.
We claim that the coadjoint operator ${\rm ad}^*_{v}$  on any exact divergence-free vector field $\xi$
is given by the Lie bracket of two fields,
\begin{equation}\label{eq:coadj}
 {\rm ad}^*_{v}\,\xi=[\xi,v]\,,
\end{equation}
just like the adjoint action. Indeed, the action of   ${\rm ad}^*_{v}$  is
as follows: for any $v,\,w\in\mathfrak X(M)$ and $\xi\in\mathfrak X_{ex}(M)$, we have
\begin{equation}
\langle w,{\rm ad}^*_{v}\xi\rangle
=\langle w,{\rm ad}^*_{v}\sigma(\xi)\rangle=
\langle w, L_{v}\sigma(\xi)\rangle=\langle w, \sigma(L_{v}\xi)\rangle=\langle w, \sigma([\xi,v])\rangle=\langle w, [\xi,v]\rangle\,.
\end{equation}

The first equality is due to the identification of spaces $\mathfrak X_{ex}$ and $\mathfrak X_{ex}^*$. (Abusing notations we use the same symbol ${\rm ad}^*$ for the corresponding fields and forms.) We also used here that the operator $\sigma$ commutes with the volume-preserving changes of coordinates, and hence with $L_v$.

The above consideration is a manifestation of the fact that the adjoint and coadjoint actions of the group ${\rm SDiff}(M)$
are geometric, i.e. they consist of the volume-preserving changes of coordinates. In the adjoint action one changes coordinates
for a vector field. For the coadjoint action one changes coordinates in the coset of 1-forms, or in the exact 2-form (the
differential of the coset), or in the (vorticity) vector field, which is the kernel field, i.e. the exact divergence-free vector field naturally related to the 2-form.

To summarize, in the case of $M$ with $H_1(M, \mathbb R)=0$, one has the identifications
$\mathfrak X^*\cong\mathfrak X^*_{ex}\cong\mathfrak X_{ex}\cong\mathfrak X$ with the Lie-bracket action,
and hence all regular integral invariants on the dual space, i.e. all regular Casimirs, are functions of the helicity. This proves
Corollary \ref{C:main}.

In the case of $M$ with nontrivial $H_1(M, \mathbb R)$, one  has a natural projection $\pi: \mathfrak X^*
\to \mathfrak X^*_{ex}=\mathfrak X_{ex}$. The helicity is well-defined for the image of this projection, i.e. for
the vorticity field, and hence it has the same value for all fields in the same coset $[u]\in \mathfrak X^*$.
But now the helicity cannot be the only invariant of the coadjoint action.
For instance, take two different {\it closed} $1$-forms $u_1$ and $u_2$ on $M$.
For their cosets $[u_i]\in \mathfrak X^*$, $i=1,2,$ one has
$\pi([u_i])=0$, i.e. they belong to the zero fiber of this projection $\pi$. Their vorticity fields vanish, $\xi_i=0$, since $du_i=0$, and hence their helicity vanish as well, $\mathcal H(\xi_i)=0$. On the other hand, the cohomology classes of $u_i$  are elements of $H^1(M, \mathbb R)$ and are invariant under coordinate transformations (in the identity connected component) of the group ${\rm SDiff}(M)$.

Again, one expects that the helicity value $\mathcal H$ and the cohomology class $ H^1(M, \mathbb R)$
describe the full set of regular invariants of the coadjoint action on $ \mathfrak X^*$.
Note,  that the fibers of the  projection $\pi: \mathfrak X^* \to \mathfrak X_{ex}$
(unlike the fibers for $\lambda:\mathfrak X\to  H_1(M, \mathbb R)$)
are finite-dimensional: they are affine spaces isomorphic to $ H^1(M, \mathbb R)$. However,
the cohomology class is well-defined only for the zero fiber of this projection, i.e. cosets $[u]$ of closed forms $u$.

\medskip

%%%%%%%%%%%%%%%%%%

\section{Casimirs in 3D ideal and magnetic hydrodynamics}\label{sec:settings}

\subsection{Geometry of the incompressible Euler equations}\label{sec:SDiff}

In \cite{Arn66} Arnold proved that the Lagrangian description of the Euler equations~\eqref{idealEuler} can be regarded as the geodesic flow on the infinite-dimensional Lie group ${\rm SDiff}(M)$
of volume-preserving diffeomorphisms of $M$ with respect to a right-invariant $L^2$-metric on the group.
This geometric point of view implies  the following Hamiltonian framework for ideal fluids.

Consider the (regular) dual space $\mathfrak X^*$
to the space $\mathfrak X $ of divergence-free vector fields on $M$, which is the Lie algebra of the group ${\rm SDiff}(M)$.
This dual space $\mathfrak X^*$ is isomorphic to
the  space of cosets  $\Omega^1(M) / d \Omega^0(M)$, where $\Omega^k(M)$ is the space of all smooth $k$-forms on $M$.
An element in $\Omega^1(M) / d \Omega^0(M)$ is $[\alpha]=\{\alpha+df\,:\,\text{ for all } f\in C^\infty(M)\}.$
The natural pairing between the arbitrary elements $[\alpha]\in\mathfrak X^*$ and $v\in\mathfrak X$ is given by
$\langle [\alpha],u\rangle:=\int_M\alpha(u)\,d\mu$, where $d\mu$ is the volume form on the manifold $M$.
The coadjoint action of the group $\rm SDiff(M)$ on the dual
 $\mathfrak X^*$ is given by the change of coordinates in (cosets of) 1-forms on $M$
 by means of volume-preserving diffeomorphisms.

Using the Riemannian metric $(\,,)$ on the manifold $M$, one can
identify the Lie algebra and its (regular) dual by means of the so-called inertia operator:
given a vector field $v$ on $M$  one defines the 1-form $v^\flat$
as the pointwise inner product with  the velocity field $v$:
$v^\flat(w): = (v,w)$ for all $w\in T_xM$. Note also  that the 1-form $v^\flat$ corresponding to a divergence-free field $v$ is co-closed.
Then the Euler equations~\eqref{idealEuler} can be rewritten on 1-forms $\alpha=v^\flat$ as
$\partial_t \alpha+L_v \alpha=-dP$
for an appropriate function $P$ on $M$, which is not the pressure (see~\cite{arkh,PS}).

In terms of the cosets of 1-forms $[\alpha]$, the Euler equations assume the form
\begin{equation}\label{1-forms}
\partial_t [\alpha]+L_v [\alpha]=0\,
\end{equation}
on the dual space $\mathfrak X^*$. Equation~\eqref{1-forms} on $\mathfrak X^*$ turns out to be Hamiltonian with respect to
the natural linear Lie-Poisson bracket on the dual space and the
Hamiltonian functional $H$ given by the kinetic energy of the fluid,
$H([\alpha])= \frac 12\int_M(v,v)\,d\mu$
for $\alpha=v^\flat$, see details in~\cite{Arn66, arkh}.
The corresponding Hamiltonian operator is given by the Lie algebra coadjoint action ${\rm ad}^*_v$,
which  in the case of the diffeomorphism group corresponds to the Lie derivative: ${\rm ad}^*_v=L_v$.
Its symplectic leaves are coadjoint orbits of the corresponding group ${\rm SDiff}(M)$.

All invariants of the coadjoint action, also called \emph{Casimir functions}, are first integrals of the Euler equations
for any choice of the Riemannian metric (with a fixed volume form). Theorem~\ref{T:main} shows that the helicity is the only $C^1$ (continuously differentiable) Casimir function for the coadjoint action of the volume-preserving diffeomorphism group $\text{SDiff}(M)$ on exact divergence-free vector fields. In the particular case that the first homology group of $M$ is trivial, this corresponds to the coadjoint action of $\text{SDiff}(M)$ on the whole dual space $\mathfrak X^*$ of the Lie algebra $\mathfrak X$ (see Section~\ref{sec:pfCor} for details).

\subsection{Magnetic- and cross-helicity in magnetohydrodynamics}\label{sec:mhd}

Combining Theorem \ref{T:main} with the results of \cite{KhPeYa} one can describe regular invariants for magnetohydrodynamics (MHD).
Recall that the MHD equations on a closed three-dimensional Riemannian manifold  $M$ describe the motion of an infinitely conducting ideal fluid of velocity $v$ carrying a magnetic field $B$:
\begin{equation}\label{eq:mag_eq}
\left\{
  \begin{array}{l}
         \partial_t v = -\nabla_vv+(\text{curl}\; B)\times B-\nabla p\,,\\
	\partial_t B = -[v,B]\,,\\
		            \text{div}\; B=  \text{div}\; v=0\,.
\end{array} \right.
\end{equation}
Here $[v,B]$ stands for the Lie bracket of the vector fields $v$ and $B$, and $\times$ denotes the cross product on $M$.

Consider the space $\mathfrak X_{ex}^k\times \mathfrak X_{ex}^k$ of pairs $(\omega, B)$ of vorticity and magnetic fields
on the manifold $M$. Then the {\it magnetic helicity} is defined as
$$
\mathcal H(B):=\int_M B\cdot {\rm curl}^{-1} B\,d\mu\,,
$$
while the {\it cross-helicity } is
$$
\mathcal H(\omega,B):=\int_M B\cdot {\rm curl}^{-1} \omega\,d\mu=\int_M B\cdot v\,d\mu\,,
$$
where $\omega ={\rm curl}\,v$ on $M$ (in other words, $v$ is the only field in $\mathfrak X_{ex}^k$ such that ${\rm curl}\,v=\omega$).

It is well known~\cite{arkh} that both the magnetic helicity $\mathcal H(B)$ and the cross-helicity $\mathcal H(\omega,B)$
are first integrals of the MHD equations. Furthermore, they are Casimirs of the MHD equations, i.e., they are invariants
of the coadjoint action of the semidirect-product group $G={\rm SDiff}(M)\ltimes \mathfrak X^*$ on the dual space
of its Lie algebra.

One can introduce the notion of a regular  integral invariant
$\mathcal F : \mathfrak X^k_{ex}\times\mathfrak X^k_{ex}\rightarrow\mathbb R$ similar to Definition~\ref{def:regInt1},
see details in \cite{KhPeYa}. Theorem \ref{T:main} has the following MHD analogue, proved mutatis mutandis.

\begin{theorem}\label{thm:main}
Let $\mathcal F $ be a regular integral invariant on $\mathfrak X^k_{ex}\times\mathfrak X^k_{ex}$, then $\mathcal F $ is a function
of the magnetic helicity and the cross-helicity provided that $k\geq 4$ (this includes the case $k=\infty$). More precisely, a $C^1$ function $f : \mathbb R\times\mathbb R \rightarrow \mathbb R$ exists such that $\mathcal F (\omega,B) =f(\mathcal H(B), \mathcal H(\omega,B))$, where $(\omega,B)\in\mathfrak X^k_{ex}\times\mathfrak X^k_{ex}$.
\end{theorem}

%%%%%%%%%%%%%%%%%%%%%%%%%%%%%%%%%%%%%%%


\begin{thebibliography}{aa}

\bibitem{AR67}
R. Abraham, J. Robbin, Transversal Mappings and Flows. Benjamin, New York, 1967.

\bibitem{Arn66} V.I. Arnold, Sur la g{\'e}om{\'e}trie diff{\'e}rentielle des groupes de {L}ie de
  dimension infinie et ses applications {\`a} l'hydrodynamique des fluides
  parfaits. Ann. Inst. Fourier 16 (1966) 319--361.

\bibitem{Arn69} V.I. Arnold, On the one-dimensional cohomology of the Lie algebra of
divergence-free vector fields and rotation numbers of dynamical
systems. Funct. Anal. Appl. 3 (1969) 319--321.

\bibitem{arkh}  V.I. Arnold, B.A. Khesin, Topological Methods in Hydrodynamics. Springer, New York, 1998.

\bibitem{Besa} M. Bessa, A generic incompressible flow is topological mixing. C. R. Math. Sci. Paris 346 (2008) 1169--1174.

\bibitem{EnPeTo} A. Enciso, D. Peralta-Salas, F. Torres de Lizaur, Helicity is the only integral invariant of
volume-preserving transformations. Proc. Natl. Acad. Sci. USA 113 (2016) 2035--2040.

\bibitem{GuHo} J. Guckenheimer, P. Holmes, Nonlinear Oscillations, Dynamical Systems, and Bifurcations of Vector Fields. Springer-Verlag, New York, (1990).

\bibitem{khiz} A. Izosimov, B. Khesin,  Classification of Casimirs in 2D hydrodynamics. Moscow Math. J. 17 (2017) 699--716.

\bibitem{KhPeYa} B. Khesin, D. Peralta-Salas, C. Yang, A basis of Casimirs in 3D magnetohydrodynamics. Int. Math. Res. Notices (2020), 16pp., doi:10.1093/imrn/rnz393.

\bibitem{Ku14} E.A. Kudryavtseva, Conjugation invariants on the group of area-preserving diffeomorphisms of the disk.
Math. Notes 95 (2014) 877--880.

\bibitem{Ku16} E.A. Kudryavtseva, Helicity is the only invariant of incompressible flows whose derivative is continuous
in the $C^1$ topology. Math. Notes 99 (2016) 611--615.

\bibitem{Mo61}
J.J. Moreau,  Constantes d'un ilot tourbillonnaire en fluid parfait barotrope. C. R. Acad. Sci. Paris 252 (1961) 2810--2812.

\bibitem{Mo69}
H.K. Moffatt, The degree of knottedness of tangled vortex lines. J. Fluid Mech. 35 (1969) 117--129.

\bibitem{Mou76}
R. Moussu, Sur l'existence d'int\'egrales premi\`eres pour un germe de forme de Pfaff.
Ann. Inst. Fourier 26 (1976) 171--220.

\bibitem{PM}
J. Palis, W. de Melo, Geometric Theory of Dynamical Systems. Springer-Verlag, New York, 1982.

\bibitem{PS}
D. Peralta-Salas, Selected topics on the topology of ideal fluid flows. Int. J. Geom. Meth. Mod. Phys. 13 (2016) 1630012.

\bibitem{Stern}
S. Sternberg, On differential equations on the torus. Amer. J. Math. 79 (1957) 397--402.

\bibitem{Tres}
D. Treschev, Volume preserving diffeomorphisms as Poincar\'e maps for volume preserving flows. Russ. Math. Surv., in press (2020).

\end{thebibliography}
\end{document}